\documentclass{article}
\usepackage{amsfonts}
\usepackage{amsmath}
\usepackage[mathscr]{eucal}
\usepackage{amssymb}
\usepackage{latexsym}
\usepackage{amsthm}
\usepackage{amscd}
\usepackage{epsf}
\usepackage{graphicx}
\usepackage{makeidx}
\usepackage{enumerate}
\usepackage{mathdots}

\newtheorem{theorem}{Theorem}[section]
\newtheorem{proposition}[theorem]{Proposition}

\newtheorem{lemma}[theorem]{Lemma}

\newtheorem{remark}[theorem]{Remark}

\newcommand{\bsx}{\boldsymbol{x}}
\newcommand{\bsy}{\boldsymbol{y}}

\def\bN{\mathbb N}

\def\bF{{\mathbb F}}

\def\bsx{{\boldsymbol{x}}}
\def\bsy{{\boldsymbol{y}}}

\newcommand{\Lset}[1][m]{\mathcal{L}_{#1}}
\newcommand{\Uset}[1][m]{\mathcal{U}_{#1}}

\newcommand{\Ftwo}{\mathbb{F}_2}
\newcommand{\FtwoMat}{\mathbb{F}_2^{m \times m}}

\newenvironment{enuroman}{\begin{enumerate}[\normalfont (i)]}{\end{enumerate}}

\begin{document}

\title{
Characterization of digital $(0,m,3)$-nets and digital $(0,2)$-sequences in base $2$\thanks{subclass 11K31, 11K38}}

\author{
Roswitha Hofer\thanks{Institute of Financial Mathematics and Applied Number Theory, Johannes Kepler University Linz, Altenbergerstr. 69, 4040 Linz, Austria. e-mail: roswitha.hofer@jku.at}  and 
Kosuke Suzuki\thanks{Graduate School of Science, Hiroshima University. 1-3-1 Kagamiyama, Higashi-Hiroshima, 739-8526, Japan. JSPS Research Fellow. e-mail: kosuke-suzuki@hiroshima-u.ac.jp}
}

\date{\today}
\maketitle

\begin{abstract}
We give a characterization 
of all matrices $A,B,C \in \FtwoMat$
which generate a $(0,m,3)$-net in base $2$ and a characterization of all matrices $B,C\in\bF_2^{\bN\times\bN}$ which generate a $(0,2)$-sequence in base $2$.
\end{abstract}

\section{Introduction and main results}
The algorithms for constructing digital $(t,m,s)$-nets and digital $(t,s)$-sequences, which were introduced by Niederreiter \cite{Niederreiter1987pss}, are well-established methods to obtain low-discrepancy point sets and low-discrepancy sequences. Low-discrepancy point sets and sequences are the main ingredients of quasi-Monte Carlo quadrature rules for numerical integration (see for example \cite{ Dick2010dna,Niederreiter1992rng} for details).
The purpose of this paper is to characterize digital nets and sequences in base $2$ with best possible quality parameter $t$. We start the paper with introducing the algorithm for digital $(t,m,s)$-nets and digital $(t,s)$-sequences in base $2$ and defining the quality parameter $t$. 

Let $\bN$ be the set of all positive integers,
$\Ftwo = \{0,1\}$ be the field of two elements.
For a positive integer $m$, $\FtwoMat$ denotes the set of all $m \times m$ matrices over $\Ftwo$.
For a nonnegative integer $n$, we write the $2$-adic expansion of $n$
as $n = \sum_{i=1}^\infty z_i(n) 2^{i-1}$ with $z_i(n) \in \Ftwo$,
where all but finitely many $z_i(n)$ equal zero.
For $k \in \bN \cup \{\infty\}$,
we define the function $\phi_k \colon \Ftwo^k \to [0,1]$ as
\[
\phi_k((y_1, \dots, y_{k})^\top) := \sum_{i=1}^k \frac{y_{i}}{2^i}.
\]

Let $s$, $m\in\bN$, and $C_1, \dots, C_s \in \FtwoMat$. The digital net generated by $(C_1, \dots, C_s)$ is a set of $2^m$ points in $[0,1)^s$
that is constructed as follows.
We define $\bsy_{n,j} \in \Ftwo^m$ for $0 \leq n < 2^m$ and $1 \leq j \leq s$ as
\[
\bsy_{n,j} := C_j \cdot (z_1(n), \dots, z_m(n))^\top \in \Ftwo^m.
\]
Then we obtain the $n$-th point $\bsx_n$ by applying $\phi_m$ componentwise to the $\bsy_{n,j}$, i.e.,
\begin{equation*}
\bsx_n := (\phi_m(\bsy_{n,1}), \dots, \phi_m(\bsy_{n,s})).
\end{equation*}
Finally letting $n$ range between $0$ and $2^m-1$ we obtain the point set $\{\bsx_0, \dots, \bsx_{2^m-1}\} \subset [0,1)^s$ that is called the digital net generated by $(C_1, \dots, C_s)$. 

In a similar way,
for $C_1, \dots, C_s \in \Ftwo^{\bN \times \bN}$ 
the digital sequence generated by $(C_1, \dots, C_s)$ is the sequence of points in $[0,1]^s$ that is constructed as follows.
We define $\bsy_{n,j} \in \Ftwo^\bN$ for $n\in\bN\cup\{0\}$ and $1 \leq j \leq s$ as
\(
\bsy_{n,j} := C_j \cdot(z_1(n), z_2(n) ,\dots)^\top \in \Ftwo^\bN.
\)
This matrix-vector multiplication is well-defined
since almost all $z_i(n)$ equal zero.
Then we obtain the $n$-th point $\bsx_n$ by setting
\(
\bsx_n := (\phi_\infty(\bsy_{n,1}), \dots, \phi_\infty(\bsy_{n,s})).
\)
The digital sequence generated by $(C_1, \dots, C_s)$
is the sequence of points $\{\bsx_0, \bsx_1, \dots \} \subset [0,1]^s$.

The nonnegative integer $t$ in the notions of $(t,m,s)$-nets and $(t,s)$-sequences quantifies in a certain sense the uniformity of digital nets and sequences. 
A set $\mathcal{P}$ of $2^m$ points in $[0,1)^s$ is said to a $(t,m,s)$-net in base $2$
if every subinterval of the form
\[
\prod_{i=1}^s [a_{i}/2^{c_{i}}, (a_{i} +1)/2^{c_{i}})  \qquad \text{with integers $c_i \geq 0$ and $0 \leq a_i < 2^{c_i}$}
\]
and of volume $2^{t-m}$ contains exactly $2^t$ points from $\mathcal{P}$.
For the definition of $(t,s)$-sequences in base $2$,
we need to introduce the truncation operator.
For $x \in [0,1]$ with the prescribed $2$-adic expansion $x = \sum_{i=1}^\infty {x_{i}}/{2^i}$
(where the case $x_i = 1$ for almost all $i$ is allowed),
we define the $m$-digit truncation $[x]_m := \sum_{i=1}^m {x_{i}}/{2^i}$.
For $\bsx = (x_1, \dots, x_s) \in [0,1]^s$,
the coordinate-wise $m$-digit truncation of $\bsx$ is defined as
$[\bsx]_m := ([x_1]_m, \dots, [x_s]_m)$.
A sequence $\mathcal{S} = \{\bsx_0, \bsx_1, \dots \}$ of points in $[0,1]^s$ with prescribed $2$-adic expansions
is said to be a $(t,s)$-sequence in base $2$
if, for all nonnegative integers $k$ and $m$, the set $\{[\bsx_{k2^m}]_m, \dots, [\bsx_{(k+1)2^m-1}]_m\}$ is a $(t,m,s)$-net in base $2$. Straightforward, we define $(t,m,s)$-nets and $(t,s)$-sequences in base $b$ with $b\in\bN\setminus\{1\}$ by substituting $2$ by $b$ in the definitions above. 

By the definitions of $(t,m,s)$-nets and $(t,s)$-sequences,
a smaller $t$ implies more conditions on the uniformity of the points and of the sequences.
Indeed a smaller $t$ corresponds with a smaller discrepancy bound (cf. \cite{Niederreiter1987pss}). Hence smaller $t$ would be appreciated and $t=0$ is the best possible.
Having lowest possible value $0$ for $t$ has another merit:
the randomized quasi-Monte Carlo estimator of a scrambled $(0,m,s)$-net in base $b$
is asymptotically normal \cite{Loh}.
However,
$t=0$ cannot be attained when $s$ is large.
It is well known that $(0,m,s)$-nets in any base $b$ exist only if $s \leq b+1$
and $(0,s)$-sequences in base $b$ exist only if $s \leq b$ \cite[Corollary~4.24]{Niederreiter1992rng}.
On the other hand, there are many known digital $(0,b)$-sequences in prime base $b$,
including the two-dimensional Sobol$^\prime$ sequence for $b=2$ \cite{Sobolcprime1967dpi},
Faure sequences \cite{Faure1982dds},
generalized Faure sequences \cite{Tezuka}, and its reordering \cite{FaureTezuka}.
From these sequences we can construct digital $(0,m,b+1)$-nets in base $b$,
see \cite[Lemma~4.22]{Niederreiter1992rng} or Lemma~\ref{lem:seq-to-net}. 

A characterization of $(0,m,3)$-net in base $2$ generated by $(I,B,B^2)$ with some $B \in \FtwoMat$
was given in \cite{Kajiura} and a characterization of $(0,2)$-sequences in base $2$ generated by NUT matrices $(C_1,C_2)$ was given in \cite{Hofer2010edc}.
Our contribution in this note is to characterize all generating matrices
of digital $(0,m,3)$-nets and digital $(0,2)$-sequences in base $2$.

For the statements of our results, we introduce some notation.
Let $I_m$ be the $m \times m$ identity matrix in $\bF_2^{m\times m}$.
Let $J_m$ be the $m \times m$ anti-diagonal matrix in $\bF_2^{m\times m}$ whose anti-diagonal entries are all $1$,
and $P_m$ be the $m \times m$ upper-triangular Pascal matrix in $\bF_2^{m\times m}$, i.e.,
\[
J_m =
\begin{pmatrix}
0 &  & 1\\
& \iddots \\
1 & &  0
\end{pmatrix},
\qquad
P_m = \left(\binom{j-1}{i-1}\right)_{i,j=1}^m =
\begin{pmatrix}
\binom{0}{0} & \binom{1}{0} & \dots & \binom{m-1}{0}\\
 & \binom{1}{1} & & \vdots   \\
 & & \ddots & \vdots  \\
 & &  & \binom{m-1}{m-1}
\end{pmatrix},
\]
which are considered modulo $2$.
If there is no confusion, we omit the subscripts
and simply write $I$, $J$, and $P$.
Let $\Lset$ (resp.\ $\Uset$) be the set of non-singular lower- (resp.\ upper-) triangular $m \times m$ matrices
over $\Ftwo$.
Note that $\Lset \cap \Uset = \{I\}$ holds.
Let  $\Lset[\infty]$ (resp.\ $\Uset[\infty]$) be the set of non-singular lower (resp.\ upper) triangular infinite matrices over $\Ftwo$.
Let $P_\infty$ be the infinite Pascal matrix,
i.e., whose $m \times m$ upper-left submatrix is $P_m$ for all $m \geq 1$.
Note that for $C\in\bF_2^{\bN\times\bN}$ and $L\in\Lset[\infty],\,U\in\Uset[\infty]$ the products $LC$ and $CU$ are well defined and $(LC)U=L(CU)$. 
For a finite or infinite matrix $C$ and for $k \in \bN$ we write $C^{(k)} \in \Ftwo^{k \times k}$ for the upper left $k \times k$ submatrix of $C$.

We are now ready to state our main results.
\begin{theorem} \label{thm:char-0m3met-general}
Let $m \geq 1$ be an integer and $A,\,B,\,C \in \FtwoMat$.
Then the following are equivalent.
\begin{enuroman}
\item \label{eq:0m3net-general-equiv1}
$(A,B,C)$ generates a digital $(0,m,3)$-net in base $2$.
\item \label{eq:0m3net-general-equiv2}
There exist $L_1, L_2 \in \Lset$, $U \in \Uset$, and non-singular $M \in \FtwoMat$
such that
\[
(A,B,C) = (JM, L_1UM, L_2PUM).
\]
\end{enuroman}
\end{theorem}

\begin{theorem}\label{thm:char-02seq}
Let $B,C\in\bF_2^{\bN\times\bN}$. Then the following are equivalent. 
\begin{enuroman}
\item \label{item:thmBC}
$(B,C)$ generates a digital $(0,2)$-sequence in base $2$.
\item \label{item:thmLPU}
There exist $L_1,L_2\in\mathcal{L}_\infty$ and $U\in\mathcal{U}_\infty$ such that $B=L_1U$ and $C=L_2P_\infty U$. 
\end{enuroman}
\end{theorem} 

In the rest of the paper, we give auxiliary results in Section~\ref{sec:auxiliary}
and prove the above theorems in Section~\ref{sec:0m3net}.
\section{Auxiliary results}\label{sec:auxiliary}

We start with $t$-value-preserving operations.
\begin{lemma}[{\cite[Lemma~2.2]{Kajiura}}]\label{lem:tval-invariant}
Let $C_1, \dots, C_s \in \FtwoMat$ and $L_1, \dots, L_s \in \Lset$.
Let $G \in \FtwoMat$ be non-singular.
Then the following are equivalent.
\begin{enuroman}
\item \label{eq:t-preserving-net-equiv1}
$(C_1, \ldots, C_s)$ generates a digital $(t,m,s)$-net.
\item \label{eq:t-preserving-net-equiv2}
$(L_1C_1G, \ldots, L_sC_sG)$ generates a digital $(t,m,s)$-net.
\end{enuroman} 
\end{lemma}

\begin{lemma}\label{lem:tval-invariant-sequence}
Let $C_1, \dots, C_s \in \Ftwo^{\bN \times\bN}$, $L_1, \dots, L_s \in \Lset[\infty]$ and $U \in \mathcal{U}_\infty$. Then the following are equivalent.
\begin{enuroman}
\item \label{eq:t-preserving-equiv1}
$(C_1, \ldots, C_s)$ generates a digital $(t,s)$-sequence.
\item \label{eq:t-preserving-equiv2}
$(L_1C_1U, \ldots, L_sC_sU)$ generates a digital $(t,s)$-sequence.
\end{enuroman} 
\end{lemma}

\begin{proof}
A slight adaption of the proof of \cite[Proposition~1]{FaureTezuka}
(resp.\ \cite[Theorem~1]{Tezuka})
shows that multiplying $L_i$ from left (resp.\ multiplying $U$ from right) does not change the $t$-value.
Note that here we used that $L_i^{-1}$ exists in $\Lset[\infty]$
and $U^{-1}$ exists in $\Uset[\infty]$.
\end{proof}

The following results point out relations between digital nets and sequences.
\begin{lemma}[{\cite[Lemma~4.22]{Niederreiter1992rng}}]\label{lem:seq-to-net}
Let $\{\bsx_i\}_{i \geq 0}$ be a $(t,s)$-sequence in base $2$.
Then $\{(\bsx_i, i 2^{-m})\}_{i=0}^{2^m -1}$ is a $(t,m,s+1)$-net in base $2$.
\end{lemma}

\begin{lemma}\label{lem:BCandJBC}
Let $C_1, \dots, C_s \in\bF_2^{\bN\times\bN}$.
Then the following are equivalent.
\begin{enuroman}
\item\label{item:BC}
$(C_1, \dots, C_s)$ generates a digital $(t,s)$-sequence.
\item \label{item:Cm}
$(C_1^{(m)}, \dots, C_s^{(m)})$ generates a digital $(t,m,s)$-net for every $m\in\mathbb{N}$.
\item \label{item:JBC}
$(J_m, C_1^{(m)}, \dots, C_s^{(m)})$ generates a digital $(t,m,s+1)$-net for every $m\in\mathbb{N}$.
\end{enuroman} 
\end{lemma}

\begin{proof}
\eqref{item:Cm} implies \eqref{item:BC} by \cite[Theorem~4.36]{Niederreiter1992rng}.
Clearly \eqref{item:JBC} shows \eqref{item:Cm}.
\eqref{item:BC} implies \eqref{item:JBC} by Lemma~\ref{lem:seq-to-net}. 
\end{proof}

Having $t=0$ is related to LU decomposability.
In particular, we have a characterization of digital $(0,1)$-sequences and digital $(0,m,2)$-nets.
\begin{lemma}\label{lem:IB-0m2net-structure}
Let $B \in \FtwoMat$. Then
$(J_m,B)$ generates a digital $(0,m,2)$-net if and only if
there exist $L \in \Lset$ and $U \in \Uset$ such that $B = LU$.
\end{lemma}
\begin{proof}
This is essentially proved in \cite[Lemma~3.1]{Kajiura}.
\end{proof}

\begin{lemma} \label{lem:char-0m2met-general}
Let $m \geq 1$ be an integer and $A,B \in \FtwoMat$.
Then $(A,B)$ generates a $(0,m,2)$-net if and only if 
there exist $L \in \Lset$, $U \in \Uset$ and non-singular $M \in \FtwoMat$
such that
$(A,B) = (JM, LUM).$
\end{lemma}

\begin{proof}
This lemma can be reduced to Lemma~\ref{lem:IB-0m2net-structure}
by a similar argument to the one in Section~\ref{sec:0m3net} that reduces Theorem~\ref{thm:char-0m3met-general} to Proposition~\ref{prop:char-0m3met}.
\end{proof}

\begin{lemma}\label{lem:BLUinfty}
Let $B \in \bF_2^{\bN\times\bN}$.
Then $B$ generates a digital $(0,1)$-sequence if and only if there exist $L\in\mathcal{L}_\infty$ and $U\in\mathcal{U}_\infty$ such that 
$B=LU.$
\end{lemma}

\begin{proof}
First we assume that $B$ generates a digital $(0,1)$-sequence.
Then it follows from Lemma~\ref{lem:BCandJBC}
that $(J_m,B^{(m)})$ generates $(0,m,2)$-net for every $m\in\mathbb{N}$.
Thus by Lemma~\ref{lem:IB-0m2net-structure} there exist $L_m \in\mathcal{L}_m$ and $U\in\mathcal{U}_m$ such that 
$B^{(m)}=L_mU_m.$
By comparing the upper left $n \times n$ submatrix of this equation for $n \leq m$,
we have $L_m^{(n)} = L_n$ and $U_m^{(n)} = U_n$ for all $n \leq m$.
This implies that there exists unique $L \in \Lset[\infty]$ and $U \in \Uset[\infty]$
such that $L^{(m)} = L_m$ and $U^{(m)} = U_m$ holds for all $m$,
and hence we have $B=LU$. This shows the ``only if'' part.
The converse holds from Lemma~\ref{lem:tval-invariant-sequence}.
\end{proof}

\begin{remark}{\rm
Lemmas~\ref{lem:tval-invariant}--\ref{lem:BLUinfty} with appropriate modifications hold for digital nets and sequences over an arbitrary finite field
since the proofs are based on general linear algebra.
The results below uses that the base field is $\Ftwo$.}
\end{remark}

The first author and Larcher essentially determined all digital $(0,2)$-sequences in base $2$
generated by non-singular infinite upper-triangular matrices \cite[Proposition~4]{Hofer2010edc}.

\begin{lemma}\label{lem:alternative}
Let $U_1, U_2 \in\mathcal{U}_m$. Then $(J_m,U_1,U_2)$ generates $(0,m,3)$-net in base $2$
if and only if $U_2=P_mU_1$ holds. 
\end{lemma}
\begin{proof}
The ``only if'' part is essentially derived in the proof of \cite[Proposition~4]{Hofer2010edc}.
Now we assume $U_2=P_mU_1$.
From the construction in \cite{Faure1982dds} and Lemma~\ref{lem:seq-to-net},
$(J,I,P)$ generates a $(0,m,3)$-net.
Then it follows from Lemma~\ref{lem:tval-invariant} with $(L_1,L_2,L_3) = (JU_1^{-1}J,I,I)$ and $G=U_1$
that $((JU_1^{-1}J)JU_1,IU_1,PU_1) = (J,U_1,PU_1)$ also generates a $(0,m,3)$-net. 
Thus we have proved the converse.
\end{proof}

\begin{proposition}
Let $U_1, U_2 \in \Uset[\infty]$. 
Then the following are equivalent:
\begin{enuroman}\label{thm:Hofer}
\item \label{eq:02seq-equiv1}
$(U_1, U_2)$ generates a digital $(0,2)$-sequence in base $2$. 
\item \label{eq:02seq-equiv2}
$U_2=P_\infty U_1$ holds.
\end{enuroman}
\end{proposition}
\begin{proof} 
\eqref{eq:02seq-equiv1} implies \eqref{eq:02seq-equiv2} by \cite[Proposition~4]{Hofer2010edc}.
The converse follows from Lemma~\ref{lem:tval-invariant-sequence} and the construction in \cite{Faure1982dds}. 
\end{proof}

\section{Proof of Theorem \ref{thm:char-0m3met-general} and \ref{thm:char-02seq}}\label{sec:0m3net}
Having all the auxiliary results of the previous section at hand, the proofs of our theorems are rather short.

\begin{proof}[Proof of Theorem~\ref{thm:char-0m3met-general}]
Let $M=JA$.
By putting $B' = BM^{-1}$ and $C' = CM^{-1}$,
$t(A,B,C)=0$ is equivalent to $t(JM,B'M,C'M)=0$,
which is equivalent to $t(J,B',C')=0$ by Lemma~\ref{lem:tval-invariant}.
Hence Theorem~\ref{thm:char-0m3met-general} reduces to the case $A=I$,
i.e., it suffices to show the following claim.
\begin{proposition} \label{prop:char-0m3met}
Let $m \geq 1$ be an integer and $B,C \in \FtwoMat$.
Then the following are equivalent.
\begin{enuroman}
\item \label{eq:0m3net-equiv1}
$(J,B,C)$ generates a $(0,m,3)$-net in base $2$.
\item \label{eq:0m3net-equiv2}
There exist $L_1, L_2 \in \Lset$ and $U \in \Uset$
such that $B=L_1U$ and $C=L_2PU$.
\end{enuroman}
\end{proposition} 
We now prove Proposition~\ref{prop:char-0m3met}.
In this proof, for matrices $Q,R,S \in \FtwoMat$
let $t(Q,R,S)$ be the $t$-value of the digital net generated by $(Q,R,S)$.
First we assume \eqref{eq:0m3net-equiv2}.
By Lemma~\ref{lem:tval-invariant} with $(L_1,L_2,L_3) = (I,L_1^{-1},L_2^{-1})$ and $G=I$
we have
\begin{equation*}\label{eq:JBC}
t(J,B,C)=t(J,L_1U,L_2PU)=t(J,U,PU) = 0.
\end{equation*}
where the last equality follows from Lemma~\ref{lem:alternative}.
Thus we have \eqref{eq:0m3net-equiv1}.

We now assume \eqref{eq:0m3net-equiv1}.
By Lemma~\ref{lem:IB-0m2net-structure},
there exist $L_1, L_2 \in \Lset$ and $U_1,U_2 \in \Uset$
such that $B=L_1U_1$ and $C=L_2U_2$.
Hence, by Lemma~\ref{lem:tval-invariant} with $(L_1,L_2,L_3) = (I,L_1,L_2)$ and $G=I$
we have
\[
t(J,U_1,U_2)=t(J,L_1U_1,L_2U_2)=t(J,B,C)=0.
\]
Finally, Lemma~\ref{lem:alternative} implies $U_2 = PU_1$, which shows \eqref{eq:0m3net-equiv2}.
\end{proof}


\begin{proof}[Proof of Theorem~\ref{thm:char-02seq}]
\eqref{item:thmLPU} implies \eqref{item:thmBC} by Lemma~\ref{lem:tval-invariant-sequence} and Proposition~\ref{thm:Hofer}. 
Let us now assume \eqref{item:thmBC}.
Then by Lemma~\ref{lem:BLUinfty}
there exist $L_1,\,L_2\in\mathcal{L}_\infty$ and $U_1,\,U_2\in\mathcal{U}_\infty$ such that $B=L_1U_1$ and $C=L_2U_2$.
We apply Lemma~\ref{lem:tval-invariant-sequence}
and obtain that $(U_1,U_2)$ generates a $(0,2)$-sequence in base $2$.
Finally Proposition~\ref{thm:Hofer} brings $U_2=P_\infty U_1$ and the result follows. 

\end{proof}

\section*{Acknowledgments}

The first author is supported by the Austrian Science Fund (FWF):
Project F5505-N26, which is a part of the Special Research Program
``Quasi-Monte Carlo Methods: Theory and Applications''. 
The second author is supported by 
Grant-in-Aid for JSPS Fellows (No.\ 17J00466).



\begin{thebibliography}{1}

\bibitem{Dick2010dna}
Josef Dick and Friedrich Pillichshammer.
\newblock {\em Digital {N}ets and {S}equences: Discrepancy {T}heory and
  {Q}uasi-Monte Carlo {I}ntegration}.
\newblock Cambridge University Press, Cambridge, 2010.

\bibitem{Faure1982dds}
Henri Faure.
\newblock Discr\'epance de suites associ\'ees \`a un syst\`eme de num\'eration
  (en dimension {$s$}).
\newblock {\em Acta Arith.}, 41(4):337--351, 1982.

\bibitem{FaureTezuka}
Henri Faure and Shu Tezuka. 
\newblock Another random scrambling of digital $(t,s)$-sequences,
\newblock in: K.T. Fang et. al (eds),
{\em Monte Carlo and Quasi-Monte Carlo Methods 2000}, Springer, Berlin:242--256, 2002. 

\bibitem{Hofer2010edc}
Roswitha Hofer and Gerhard Larcher.
\newblock
On existence and discrepancy of certain digital Niederreiter--Halton sequences.
\newblock {\em Acta Arith.}, 141(4):369--394, 2010.

\bibitem{Kajiura}
Hiroki Kajiura, Makoto Matsumoto and Kosuke Suzuki.
\newblock
Characterization of matrices $B$
such that $(I,B,B^2)$ generates a digital net
with $t$-value zero.
\newblock {\em Finite Fields Appl.}, 52:289--300, 2018.

\bibitem{Loh}
Wei-Liem Loh.
\newblock
On the asymptotic distribution of scrambled net quadrature.
\newblock 
{\em Ann. Statist.}, 31(4):1282--1324, 2003.

\bibitem{Niederreiter1987pss}
Harald Niederreiter.
\newblock {Point sets and sequences with small discrepancy.}
\newblock
{\em Monatsh. Math.}, 104(4):273--337, 1987.
  
\bibitem{Niederreiter1992rng}
Harald Niederreiter.
\newblock {\em Random number generation and quasi-{M}onte {C}arlo methods},
  volume~63 of {\em CBMS-NSF Regional Conference Series in Applied
  Mathematics}.
\newblock Society for Industrial and Applied Mathematics (SIAM), Philadelphia,
  PA, 1992.

\bibitem{Sobolcprime1967dpi}
Il'ya~Meerovich Sobol{$^\prime$}.
\newblock Distribution of points in a cube and approximate evaluation of
  integrals.
\newblock {\em \u Z. Vy\v cisl. Mat. i Mat. Fiz.}, 7:784--802, 1967.

\bibitem{Tezuka}
Shu Tezuka. 
\newblock {\em A Generalization of Faure Sequences and its Efficient Implementation}. 
\newblock Research ReportIBM RT0105:1--10, 1994. 

\end{thebibliography}

\end{document}